\documentclass[10pt]{article}
\title{Star Quasiconvexity: a Unified Approach for Linear Convergence of First-Order Methods Beyond  Convexity}
\author{Phan Quoc Khanh\thanks{Analytical and Algebraic Methods in Optimization Research Group (AAMO), Faculty of Mathematics and Statistics, Ton Duc Thang University, Ho Chi Minh City, Vietnam. Email: phanquockhanh@tdtu.edu.vn; pqkhanhus@gmail.com, ORCID-ID: 0000-0002-0524-0498}
\and
Felipe Lara\thanks{Instituto de Alta investigaci\'on (IAI), Universidad de Tarapac\'a, 
Arica, Chile. E-mail: felipelaraobreque@gmail.com; flarao@academicos.uta.cl. 
Web: felipelara.cl, ORCID-ID: 0000-0002-9965-0921} }

\usepackage{amsmath}
\usepackage{amsthm}
\usepackage{amssymb}
\usepackage{multicol}
\usepackage[active]{srcltx}
\usepackage{algorithm}
\usepackage{array}
\usepackage[dvipsnames,svgnames]{xcolor}
\usepackage[notcite,notref]{showkeys}
\usepackage{tikz}
\usepackage{subfig}
\usepackage{graphicx}
\usepackage{pgfplots}
\usepackage{xcolor}
\usepackage{ulem}  
\usepackage{comment}
\providecommand{\U}[1]{\protect \rule{.1in}{.1in}}
\newtheorem{theorem}{Theorem}

\newtheorem{corollary}[theorem]{Corollary}

\newtheorem{definition}[theorem]{Definition}
\newtheorem{example}[theorem]{Example}

\newtheorem{lemma}[theorem]{Lemma}

\newtheorem{proposition}[theorem]{Proposition}
\newtheorem{remark}[theorem]{Remark}

\begin{document}

\maketitle

\begin{abstract}
\noindent 
We introduce and study a new class of generalized convex functions termed star quasiconvex functions. This class includes convex, star-convex, quasiconvex, quasar-convex, and positively homogeneous functions of any degree $p>0$ as special cases. Furthermore, we provide several characterizations of this class covering both nonsmooth and differentiable cases. In particular, in the general nonsmooth case, star quasiconvex functions are characterized by functions for which all its sublevel sets are  star-shaped at a minimizer, while in the differentiable case strongly star quasiconvex functions coincides with those satisfying the restricted secant inequality property, and star quasiconvex functions are related with variationally coherent functions, thereby providing a rich framework for differentiable first-order methods. Additionally, we develop standard properties of the proximity operator and prove that the proximal point algorithm converges linearly to the unique solution when applied to strongly star quasiconvex functions defined over closed star-shaped sets that are not necessarily convex.

\medskip

\noindent{\small \emph{Keywords}: Nonconvex optimization; Gradient methods; Proximal point methods; Linear convergence; Star-shaped functions}

\medskip
\noindent {\bf AMS subject classifications:} 90C26; 90C30; 65K05
\end{abstract}

\section{Introduction}

Let $h: \mathbb{R}^{n} \rightarrow \overline{\mathbb{R}} := \mathbb{R} \cup \{\pm \infty\}$ be a proper function. The most classical problem in optimization is the unconstrained minimization problem:
\begin{equation}\label{problem:min}
 \min_{x \in \mathbb{R}^{n}} \, h(x). \tag{P}
\end{equation}
This paper investigates a generalized notion of convexity, specifically, a definition based purely on geometry, without imposing any additional topological properties, that guarantees linear convergence of first-order methods for solving the minimization problem \eqref{problem:min}.

The standard convexity assumption for achieving linear convergence in gradient and proximal point methods is that the objective function $h$ is strongly convex. This property is highly desirable as it guarantees the existence and uniqueness of a solution, implies quadratic growth conditions, and ensures the Polyak-$\L$ojasiewicz (PL) property, among other benefits.

Motivated by these properties, several authors have introduced broader classes of functions that retain these convergence guarantees. Notable examples are the following.

Strongly quasiconvex functions were introduced by Polyak in 1966 as a particular case of uniformly quasiconvex functions \cite{P}. This was one of the first  generalized functions possessing these guarantees. Although not immediately proven, it was later established that this class satisfies the PL property and quadratic growth, leading to linear convergence rates for gradient-type methods \cite{HLMV, LMV} and pro\-xi\-mal point methods \cite{Lara-9}.

Strongly star-convex functions were defined by Nesterov and Polyak in 2006 \cite{NP-2006}. This class also ensures linear convergence for both gradient and proximal point methods, as demonstrated in \cite{HSS} and \cite{BLL}, respectively.

Strongly quasar-convex functions were introduced under various names in \cite{GG, GGK-2017, HSS} (among others). This class has been shown to yield linear convergence for gradient-type methods \cite{GG, GGK-2017, HADR, HSS} and, recently, for the proximal point method in \cite{BLL}.

The relationship between these classes is summarized below:
 \begin{align}
  \begin{array}{ccccccc}
  {\rm strongly ~ quasiconvex} & \Longleftarrow & {\rm strongly ~ convex} 
  & \Longrightarrow & {\rm strongly ~ star-convex} \notag \\
  \, & \, & \, & \, & \Downarrow  \notag \\
  \, & \, & \, & \, & {\rm strongly ~ quasar-convex} \notag 
  \end{array}
 \end{align} 

These notions have gained increasing attention in the optimization and machine learning communities due to their interesting theoretical properties for accelerating first-order methods, as well as in several machine learning applications (see \cite{BLL,GLM-Survey,GGK-2017,HLMV,HADR,HSS,LMV,LV,WW,ZY} and references therein). 

Since all the previous notions includes strongly convex functions and they provide linear convergence for first-order type methods, the authors in \cite{HLMV,LMV,LV} studied the relation between (strongly) quasiconvex and (strongly) quasar-con\-vex functions in the differentiable case. It was shown in \cite[Subsection 3.2]{HLMV} that the two classes are not related each other and, as a consequence, the following natural question arises:

\medskip

{\it
\centerline{What is the class of generalized (strongly) convex functions} 
\centerline{that guarantees global linear convergence rate for first-order type methods and} 
\centerline{includes both (strongly) quasiconvex and (strongly) quasar-convex functions?}
}

\medskip

{\bf Our Contribution:}
We introduce the notion of {\it (strong) star qua\-si\-con\-ve\-xi\-ty} for solving the above question by proving that both (strongly) quasiconvex and (strongly) quasar-convex functions are (strongly) star quasiconvex, an inclusion which is proper since we show examples of strongly star quasiconvex functions which are neither quasiconvex nor quasar-convex (see Example \ref{4:leaf}). Second, we provide four di\-ffe\-rent characterizations for (strongly) star quasiconvex functions: (i) by cha\-rac\-te\-ri\-zing star quasiconvex functions via the star-shapedness of their sublevel sets in Theorem \ref{geo:charac}; (ii) by showing that (strongly) star quasiconvex functions are (strongly) quasiconvex functions along rays starting from their unique minimizer in Theorem \ref{main:rel}; (iii) by finding a stronger property for the continuous case in Corollary \ref{coro:Charact}; and (iv) by characterizing the behaviour of its gradient in the differentiable case in Proposition \ref{char:diff}.  In particular, as a consequence of the characterization in the differentiable case, we obtain that strongly star quasiconvex functions coincides with those satisfying the restricted secant inequality property \cite{Yi,ZY}, while star quasiconvex functions are related to variationally coherent functions \cite{ZBoyd} (see also \cite{SS}). Third, several interesting properties for smooth and nonsmooth strongly star quasiconvex functions are derived from the previous ana\-ly\-sis as: quadratic growth, the PL property, and also $2$-supercoercivity as well as relationships with the proximity operator on star-shaped sets among others. Notably, some of these properties are even novel for strong quasar-convexity and/or strong quasiconvexity. Fourth, we revisited the linear convergence results of the heavy ball and Nesterov accelerations. Fifth, we ensure linear convergence of the proximal point algorithm (PPA henceforth) for strongly star quasiconvex on star-shaped sets, extending the results in \cite{BLL,Lara-9} from strongly quasar-convex and strongly quasiconvex, respectively, to strongly star quasiconvex functions. Furthermore, as far as we know, this is the first convergence result for the PPA on star-shaped sets. Finally, we present conclusions and future research lines.

\medskip

The structure of the paper is as follows: In Section \ref{sec:02}, we recall basic definitions about generalized convexity. In Section \ref{sec:03}, we present the motivation for introducing (strong) star quasiconvex functions, four different characterizations for this class of functions as well as interesting properties for developing first-order type algorithms. In Section \ref{sec:04}, we present the linear convergence of the gradient method as well as for its heavy ball and Nesterov acce\-le\-ra\-tions. In Section \ref{sec:05}, we ensure linear convergence for the PPA on star-shaped sets. Conclusions and future research directions are described in Section \ref{sec:06}.

\section{Preliminaries}\label{sec:02}

The inner product in $\mathbb{R}^{n}$ and the Euclidean norm are denoted by $\langle \cdot,\cdot \rangle$ and $\lVert \cdot \rVert$, respectively. We denote  $\mathbb{R}_{+} := [0, + \infty[$, $ \mathbb{R}_{++} := \, ]0, + \infty[$ and $\mathbb{N}_{0} := \mathbb{N} \cup \{0\}$. Let $K \subseteq \mathbb{R}^{n}$. Its asymptotic (recession) cone is defined by
$$K^{\infty} := \left\{ u \in \mathbb{R}^{n}: ~ \exists ~ t_{k} \rightarrow + \infty, ~ \exists~ x_{k} \in K, ~ \frac{x_{k}}{t_{k}} \rightarrow u \right\}.$$ 
Furthermore, for any set $K \subset \mathbb{R}^{n}$, it follows from \cite[Proposition 2.1.2]{AT} that $K^{\infty}=\{0\}$ if and only if $K$ is bounded.

Given any $x, y, z \in \mathbb{R}^{n}$ and any $\beta \in \mathbb{R}$, the following relations hold:
\begin{align}
  \lVert \beta x + (1-\beta) y \rVert^{2} &= \beta \lVert x \rVert^{2} +
  (1 - \beta) \lVert y\rVert^{2} - \beta(1 - \beta) \lVert x - y \rVert^{2},
 \label{iden:1}\\
 \langle x - z, y - x \rangle &= \frac{1}{2} \lVert z - y \rVert^{2} -
 \frac{1}{2} \lVert x - z \rVert^{2} - \frac{1}{2} \lVert y - x \rVert^{2}.
 \label{3:points}
\end{align}

Let $K \subseteq \mathbb{R}^{n}$ be nonempty and $x_{0} \in K$.  $K$ is said to be star-shaped at the point $x_{0}$ if for every $x \in K$, the segment $[x_{0}, x] \subseteq K$.

\medskip

Given any extended-valued function $h: \mathbb{R}^{n} \rightarrow
\overline{\mathbb{R}}$, the effective domain of $h$ is defined by
${\rm dom}\,h := \{x \in \mathbb{R}^{n}: h(x) < + \infty \}$. It is said
that $h$ is proper if ${\rm dom}\,h$ is nonempty and $h(x) > - \infty$ for
all $x \in \mathbb{R}^{n}$. The notion of properness is important when
dealing with minimization pro\-blems.

The set ${\rm epi}\,h := \{(x,t) \in \mathbb{R}^{n} \times \mathbb{R}: h(x) \leq t\}$ is the epigraph of $h$, $S_{\delta} (h) := \{x \in \mathbb{R}^{n}: h(x) \leq \delta\}$ is the sublevel set of $h$ at the height $\delta \in \mathbb{R}$, and  ${\rm argmin}_{ \mathbb{R}^{n}} h$ is the set of all minimal points of $h$. A function $h$ is lower se\-mi\-continuous (lsc henceforth) at $\overline{x} \in \mathbb{R}^{n}$ if for any sequence $\{x_k\}_{k} \subseteq \mathbb{R}^{n}$ with $x_k \rightarrow \overline{x}$, $h(\overline{x}) \leq \liminf_{k \rightarrow + \infty} h(x_k)$. Furthermore, the current convention $\sup_{\emptyset} h := - \infty$ and $\inf_{\emptyset} h := + \infty$ is adopted.

A function $h$ with convex domain is said to be:
\begin{itemize}
 \item[$(a)$] (strongly) convex on ${\rm dom}\,h$ when there exists $\gamma \geq 0$ such that if for all $x, y \in \mathrm{dom}\,h$ and all $\lambda \in [0, 1]$, we have
 \begin{equation*}\label{strong:convex}
  h(\lambda y + (1-\lambda)x) \leq \lambda h(y) + (1-\lambda) h(x) - \lambda (1 - \lambda) \frac{\gamma}{2} \lVert x - y \rVert^{2},
 \end{equation*}

 \item[$(b)$] (strongly) quasiconvex on ${\rm dom}\,h$ when there exists $\gamma \geq 0$ such that if for all $x, y \in \mathrm{dom}\,h$ and all $\lambda \in [0, 1]$, we have
 \begin{equation*}\label{strong:quasiconvex}
  h(\lambda y + (1-\lambda)x) \leq \max \{h(y), h(x)\} - \lambda(1 -
  \lambda) \frac{\gamma}{2} \lVert x - y \rVert^{2}.
 \end{equation*}
\end{itemize} 
\noindent Every (strongly) convex function is (strongly) quasiconvex, while the reverse statement does not holds (see \cite{CM-Book,HKS,Lara-9}). Furthermore, when we say that a function is {\it strongly convex (resp., quasiconvex) with modulus $\gamma \geq 0$} we refer to both strongly convex (resp., quasiconvex) when $\gamma > 0$ and convex (resp., qua\-siconvex) when $\gamma = 0$.

Quasiconvexity is the mathematical formulation of the assumption {\it tendency to diversification} on the consumers in consumer's preference theory (see \cite{D-1959}). Furthermore, the following geometric characterization is well-known:
 \begin{align*}
  h ~ \mathrm{is ~ convex} & \Longleftrightarrow \, \mathrm{epi}\,h ~
  \mathrm{is ~ a ~ convex ~ set.}\\ 
  h ~ \mathrm{is ~ quasiconvex} & \Longleftrightarrow \, S_{\delta} (h) ~ \mathrm{is ~ a ~ convex ~ set ~ for ~ all ~ } \delta \in \mathbb{R}.
 \end{align*}

A proper function $h: \mathbb{R}^{n} \rightarrow \overline{\mathbb{R}}$ is said to be:
\begin{itemize}
 \item[$(i)$] 2-supercoercive, if
 \begin{equation*}\label{2:super}
  \liminf_{\lVert x \rVert \rightarrow+ \infty} \frac{h(x)}{\lVert x
  \rVert^{2}} >0,
 \end{equation*}

 \item[$(ii)$] coercive, if
 \begin{equation*}\label{coercive}
  \lim_{\lVert x \rVert \rightarrow+ \infty} h(x) = + \infty
 \end{equation*}
 or equivalently, if $S_{\delta} (h)$ is bounded for all $\delta \in \mathbb{R}$.
\end{itemize}
Clearly, $(i) \Rightarrow(ii)$, but the converse
statements does not hold as the function $h(x) = \lvert x \rvert$ shows. 

We know by \cite[Theorem 1]{Lara-9} that strongly quasiconvex functions 
are $2$-supercoercive, as a consequence, lsc strongly quasiconvex functions have a unique minimizer on a closed convex set as we recall next.

\begin{lemma}\label{exist:unique} {\rm (\cite[Corollary 3]{Lara-9})}
 Let $K \subseteq \mathbb{R}^{n}$ be a closed and convex set and $h:
 \mathbb{R}^{n} \rightarrow \overline{\mathbb{R}}$ be a proper, lsc, and
 strongly qua\-si\-con\-vex function on $K \subseteq {\rm dom}\,h$ with
 modulus $\gamma> 0$. Then, ${\rm argmin}_{K} h$ is a singleton.
\end{lemma}
Furthermore, the unique minimizer $\overline{x} \in K$ of a strongly quasiconvex function $h$ (with modulus $\gamma > 0$) satisfies a quadratic growth condition (see \cite[Corollary 9]{HL-1}):
\begin{equation*}\label{qgc}
 h(\overline{x}) + \frac{\gamma}{4} \lVert y - \overline{x} \rVert^{2} \leq h(y), ~ \forall ~ y \in K.
\end{equation*}

A function $h: \mathbb{R}^{n} \rightarrow \mathbb{R}$ is said to be 
$L$-smooth on $K \subseteq \mathbb{R}^{n}$ if it is differentiable on 
$K$ and
\begin{equation*}\label{L:smooth}
 \lVert \nabla h(x) - \nabla h(y) \rVert \leq L \lVert x - y \rVert, ~ 
 \forall ~ x, y \in K.
\end{equation*}

For $L$-smooth functions, a fundamental result is the descent lemma, that is, if $h$ is a $L$-smooth function on a convex set $K$ with $L \geq 0$, then for every $x, y \in K$, we have
\begin{equation*}\label{descent:lemma}
 h(y) \leq h(x) + \langle \nabla h(x), y - x \rangle + \frac{L}{2} \lVert  x - y \rVert^{2}.
\end{equation*}

Let $K \subseteq \mathbb{R}^{n}$ be a convex set and $h: K \rightarrow \mathbb{R}$ be a differentiable function. Then, $h$ is strongly convex on $K$ with modulus $\gamma \geq 0$ if and only if 
\begin{equation*}\label{sconvex:diff}
 h(x) \geq h(y) + \langle \nabla h(y), x - y \rangle + \frac{\gamma}{2} \lVert y - x \rVert^{2}, ~ \forall ~ x, y \in K.
\end{equation*}

For differentiable strongly quasiconvex functions, we have the following result \cite{VNC-2} (see also \cite{HL-1}).  Let $K \subseteq \mathbb{R}^{n}$ be 
a convex and open set and $h: K \rightarrow \mathbb{R}$ be differentiable function. If $h$ is strongly quasiconvex with modulus $\gamma \geq 0$, then for every $x, y \in K$, we have 
 \begin{equation}\label{gen:char}
  h(x) \leq h(y) ~ \Longrightarrow ~ \langle \nabla h(y), x - y \rangle
  \leq -\frac{\gamma}{2} \lVert y - x \rVert^{2}.
 \end{equation}
Conversely, if \eqref{gen:char} holds, then $h$ is strongly quasiconvex with modulus $\frac{\gamma}{2}$ (see \cite[Theorems 1 and 6]{VNC-2} and \cite[Corollary 8]{HL-1}).

The previous result extends the well-known characterization for differentiable 
quasiconvex functions given by Arrow-Enthoven \cite{AE} ($\gamma = 0)$.

\medskip

Another generalized convexity notion, mainly motivated by its interesting 
algorithmic properties, is the notion of star-convexity (see \cite{NP-2006}). 

\begin{definition}\label{def:starconvex} {\rm (\cite{NP-2006})}
A proper function $h: \mathbb{R}^{n} \rightarrow \overline{\mathbb{R}}$ 
with $\overline{x} \in {\rm argmin}_{\mathbb{R}^{n}}\,h$ is said to be 
(strongly) star-convex if there exists $\gamma \geq 0$ such that for every $y \in \mathrm{dom}\,h$ and every $\lambda \in [0, 1]$, we have
 \begin{equation*}\label{star:convex}
  h(\lambda \overline{x} + (1-\lambda)y) \leq \lambda h(\overline{x}) + (1 - \lambda) h(y) - \lambda (1 - \lambda) \frac{\gamma}{2} \lVert y - \overline{x} \rVert^{2}.
 \end{equation*}
\end{definition}

Motivated by the previous definition, the following generalized convexity notion has been introduced and applied in several problems related to machine learning theory in virtue of its acce\-le\-ra\-ted properties for the convergence of the gradient method (see \cite{GGK-2017,HADR,HSS,LV,WW} and references therein). 

\begin{definition}\label{def:qconvex} {\rm (\cite{GG,HSS}, \cite{BLL})}
 A proper function $h: \mathbb{R}^{n} \rightarrow \overline{\mathbb{R}}$ with $\overline{x} \in {\rm argmin}_{\mathbb{R}^{n}}\,h$ is said to be 
 (strongly) quasar-convex if there exists $\beta \in \, ]0, 1]$ and $\gamma \geq 0$ such that for every $y \in {\rm dom}\,h$, we have
 \begin{equation*}
  h(\lambda \overline{x} + (1-\lambda)y) \leq \lambda \beta h(\overline{x}) + (1 - \lambda \beta) h(y) - \lambda \left(1 - \frac{\lambda}{2 - \beta} \right) \frac{\beta \gamma}{2} \lVert y - \overline{x} \rVert^{2}. \notag
 \end{equation*}
\end{definition}

It follows from Definition \ref{def:qconvex} that ${\rm argmin}_{\mathbb{R}^{n}}\,h$ is a singleton when $h$ is strongly quasar-convex ($\gamma > 0$). Furthermore, if $\beta = 1$, then a quasar-convex function re\-duces to a
strongly star-convex function ($\gamma > 0$) and star-convex function ($\gamma = 0$). 

In the differentiable case, we have the following useful characterization for
(strongly) quasar-convex functions. Let $K \subseteq \mathbb{R}^{n}$ be a 
convex and open set and $h: \mathbb{R}^{n} \rightarrow \mathbb{R}$ be a
differentiable function with ${\rm argmin}_{K}\,h \neq \emptyset$. Then $h$ 
is (strongly) quasar-convex with modulus $\beta \in \, ]0, 1]$ and $\gamma 
\geq 0$ on $K$ if and only if (see \cite[Lemma 10]{HSS})
\begin{equation*}\label{char:qconvex}
 h(\overline{x}) \geq h(y) + \frac{1}{\beta} \langle \nabla h(y), \overline{x} 
 - y \rangle + \frac{\gamma}{2} \lVert y - \overline{x} \rVert^{2}, ~ \forall ~ 
 y \in K.
\end{equation*}
 Hence, quasi-strongly convex functions \cite{NNG} are related to strong quasar-convexity. Furthermore, differentiable quasar-convex functions ($\gamma = 0$) are the functions for which:  
\begin{equation}\label{weak:quasiconvexity}
 \langle \nabla h(y), \overline{x} - y \rangle \leq \beta (h(\overline{x}) - h(y)), 
 ~ \forall ~ y \in K.
\end{equation}
Functions satisfying relation \eqref{weak:quasiconvexity} have been studied 
in the last years in di\-ffe\-rent fields, specially for accelerating the convergence 
of gradient type methods (see \cite{ADR,CEG} and references therein). 

Finally, given a nonempty closed set $K \subseteq \mathbb{R}^{n}$, the proximity operator on $K$ of parameter $\beta > 0$ of a proper lsc  function $h: \mathbb{R}^{n} \rightarrow \overline{\mathbb{R}}$ at $x \in \mathbb{R}^{n}$ is defined as the operator $\mathrm{Prox}_{\beta h} (K, \cdot): \mathbb{R}^{n} \rightrightarrows \mathbb{R}^{n}$ by
\begin{equation*}\label{prox:operator}
 \mathrm{Prox}_{\beta h} (K, x) = \mathrm{argmin}_{y \in K} \left\{ h(y) + \frac{1}{2 \beta} \Vert y - x \rVert^{2} \right\}.
\end{equation*}
When $K = \mathbb{R}^{n}$, we simple write ${\rm Prox}_{\beta h} (z) := {\rm Prox}_{\beta h} (\mathbb{R}^{n}, z)$. If $h$ is proper, lsc and convex, then ${\rm Prox}_{\beta h}$ turns out to be a single-valued ope\-ra\-tor (see, for instance, \cite[Proposition
12.15]{BC-2}).

For a further study on star convexity, quasiconvexity, quasar-convexity and its applications in economics and engineering among others, we refer to
\cite{ADR,BLL,CEG,CM-Book,HKS,HSS,Lara-9,LV,NP-2006,VNC-2,WW} and references therein.

\section{An Unifying Approach for Generalized Convex Functions}\label{sec:03}

\subsection{Motivation}

As mentioned in the introduction, the following definition (see \cite{NTV}) includes  (strongly) convex, (strongly) quasiconvex, (strongly) star-convex, and (strongly) quasar-convex functions.

\begin{definition}\label{ss:quasiconvex}
 Let $h: \mathbb{R}^{n} \rightarrow \overline{ \mathbb{R}}$ be a proper function and $\overline{x} \in {\rm argmin}_{\mathbb{R}^{n} }\,h$. Then, $h$ is (strongly) star quasiconvex with respect to $\overline{x}$ if there exists $\gamma \geq 0$ such that for every $y \in {\rm dom}\,h$ and every $\lambda \in [0, 1]$, we have
 \begin{align}\label{ss:qcx}
  h(\lambda \overline{x} + (1-\lambda)y) \leq h(y) - \lambda (1 - \lambda) \frac{\gamma}{2} \lVert y - \overline{x} \rVert^{2}.
 \end{align}
 It is said that $h$ is strongly star quasiconvex when $\gamma > 0$ and it is star quasiconvex when $\gamma = 0$, in both cases, with respect to a specific minimizer $\overline{x}$.
\end{definition}
We note immediately that when $\gamma > 0$, strongly star quasiconvex functions have a unique minimizer. As a consequence, in this case, we simple say that $h$ is strongly star quasiconvex with modulus $\gamma > 0$, without reference to a point.

\medskip

Strongly star quasiconvex functions are an extension of strongly quasiconvex functions since in relation \eqref{ss:qcx} one of the points is fixed at the minimum (which always exists for lsc strongly quasiconvex functions by Lemma \ref{exist:unique}), and it is also an extension of strongly quasar-convex functions since the term $\lambda \beta h(\overline{x}) + (1-\lambda\beta) h(y) \leq h(y) = \max\{h(\overline{x}), h(y)\}$ for all $\lambda \in [0, 1]$ and $\beta \in \, ]0, 1]$. 

\begin{proposition}\label{inclusion}
 Let $h: \mathbb{R}^{n} \rightarrow \overline{\mathbb{R}}$ be a proper function. 
 Then the following re\-la\-tion\-ship holds (qcx = quasiconvex and qconvex = quasar-convex): 
 \begin{align}
  \begin{array}{ccccccc}
   {\rm (strongly) ~ convex} & \overset{*}{\Rightarrow} & {\rm (strongly) ~ star ~ convex} & \Rightarrow & 
   {\rm (strongly) ~ qconvex} \notag \\
   \Downarrow & \, & \Downarrow & \, & \Downarrow  \notag \\
   {\rm (strongly) ~ qcx} & \overset{*}{\Rightarrow} & {\rm (strongly) ~ star ~ qcx} & \, & {\rm (strongly) ~ star ~ qcx} \notag 
   \end{array}
  \end{align}
  where ``$*$" denotes that ${\rm argmin}_{\mathbb{R}^{n}\,}h \neq \emptyset$ 
  is required. 
%
%
%
%
\end{proposition}

\begin{proof}
 The proofs are straightforward, Hence, we only verify that (strong) quasar-convexity implies (strong) star quasiconvexity for $\gamma \geq 0$.

 Let $\beta \in \, ]0, 1]$ and $\gamma \geq 0$. Then, for every $\lambda \in [0, 1]$ we have
 \begin{align*}
  0 < \beta \leq 1 & \Rightarrow \frac{1}{2-\beta} \leq 1 \Rightarrow \, - \frac{\lambda}{2-\beta} \geq -\lambda   \Rightarrow \, - \left( 1 - \frac{\lambda}{2 - \beta} \right) \leq - \left( 1 - \lambda \right).
 \end{align*}
 Hence, it follows from Definition \ref{def:qconvex} that 
 \begin{align*}
  h(\lambda \overline{x} + (1-\lambda)y) & \leq \lambda \beta h(\overline{x}) + (1 - \lambda \beta) h(y) - \lambda \left(1 - \frac{\lambda}{2 - \beta} \right) \frac{\beta \gamma}{2} \lVert y - \overline{x} \rVert^{2} \\
  & \leq \lambda \beta h(\overline{x}) + (1 - \lambda \beta) h(y) - \lambda  (1 - \lambda) \frac{\beta \gamma}{2} \lVert y - \overline{x} \rVert^{2} \\
  & \leq h(y) - \lambda (1 - \lambda) \frac{\beta \gamma}{2} \lVert y - \overline{x} \rVert^{2}. 
 \end{align*}
 Thus, $h$ is strongly star quasiconvex with modulus $\gamma^{\prime} = \beta \gamma \geq 0$.
\end{proof}

\begin{remark}
 \begin{itemize}
  \item[$(i)$] All the reverse implications do not hold true in general. Indeed, almost all the counter-examples are well-known (see \cite{CM-Book,Lara-9,NP-2006}), and it remains to show an example of a strongly star quasiconvex function which is neither (strongly) quasiconvex nor (strongly) quasar-convex. This will be given in Example \ref{4:leaf} (below). 

  \item[$(ii)$] There is no relationship between (strongly) quasiconvex and (strong\-ly) quasar-convex functions. Indeed, the function in \cite[Example 9]{HLMV} is strong\-ly quasiconvex but not quasar-convex, while the function in Example \ref{exam:02} (below) is strongly quasar-convex but not quasiconvex. For differentiable func\-tions, sufficient conditions under which strong quasiconvexity implies quasar-convexity were provided in \cite{HLMV,LV}.

  \item[$(iii)$] From the proof of Proposition \ref{inclusion} we can observe that  both modulus involved in strong quasar-convexity $\gamma > 0$ and $\beta \in \, ]0, 1]$ influence the modulus of strong star quasiconvexity. Hence, if for instance $\beta \in \, ]0, 1]$ is very small, then the modulus of strong star quasiconvexity $\gamma^{\prime} = \beta \gamma$ will be small too, but still strictly positive.
 \end{itemize}
\end{remark}

In the next result, we characterize (strongly) star quasiconvex functions along segments. 

\begin{proposition}\label{char:onedim}
 A function $h: \mathbb{R}^{n} \rightarrow \overline{\mathbb{R}}$ is 
 strongly star quasiconvex with mo\-du\-lus $\gamma \geq 0$ if and only if for
 $\overline{x} \in {\rm argmin}_{\mathbb{R}^{n}}\,h$ and every $y \in
 \mathbb{R}^{n} \backslash \{\overline{x}\}$, the func\-tion 
 \begin{equation*}\label{eq:onedim}
  h_{y} (t) := h \left( \overline{x} + t \frac{y - \overline{x}}{\lVert y - \overline{x} \Vert} \right),
 \end{equation*}
 with $t \in [ 0, \lVert y - \overline{x} \rVert]$, is strongly star quasiconvex with the same modulus $\gamma \geq 0$.
\end{proposition}

\begin{proof}
 $(\Rightarrow)$: Let $\overline{x} \in \mathrm{argmin}_{\mathbb{R}^{n}}\,h$ and $y \in \mathbb{R}^{n} \backslash \{ \overline{x}\}$.   Note that $\mathrm{argmin}_{\mathbb{R}}\,h_{y} = \{0\}$. Hence, for $t_{1} = 0$ and $t_{2} \in [0, \lVert y - \overline{x} \rVert]$, we have
\begin{align*}
 h_{y} (\lambda0 + (1-\lambda) t_{2})  &  = h \left(  \overline{x} + (\lambda 0 + (1-\lambda) t_{2}) \frac{y - \overline{x}}{\lVert y - \overline{x} \Vert} \right) \\
 &  = h \left(  \lambda \left(  \overline{x} + 0 \right)  + (1 - \lambda) \left( \overline{x} + \frac{ t_{2}}{\lVert y - \overline{x} \Vert} (y - \overline{x}) \right)  \right) \\
 &  \leq h \left( \overline{x} + \frac{ t_{2}}{\lVert y - \overline{x} \Vert} (y - \overline{x}) \right)  - \lambda(1-\lambda) \frac{\gamma}{2} \left \lVert \frac{ t_{2}}{\lVert y - \overline{x} \Vert_{2}} (y - \overline{x}) \right 
 \rVert^{2} \\
 &  = h_{y} (t_{2}) - \lambda(1-\lambda) \frac{\gamma}{2} (t_{2} - 0)^{2},
\end{align*}
i.e., $h_{y}$ is strongly star quasiconvex on $[0, \lVert y - \overline{x} \rVert]$ with modulus $\gamma \geq 0$.

$(\Leftarrow)$: Assume that for $\overline{x} \in \mathrm{argmin}_{\mathbb{R}^{n}}\,h$ and every $y \in \mathbb{R}^{n} \backslash \{ \overline{x} \}$, $h_{y}$ is strongly star quasiconvex with modulus  $\gamma \geq 0$ on $[0, \lVert y - \overline{x} \rVert]$. Note that $0 \in {\rm argmin}_{ \mathbb{R}^{n} }\,h_{y}$. For every $t \in [0,1]$, we write $t \lVert y - \overline{x} \rVert = (1-t)0 + t \lVert y - \overline{x} \rVert$ and use the strong star quasiconvexity of $h_{y}$ at the points $0$, $\lVert y - \overline{x} \rVert$, and $(1-t) 0 + t \lVert y - \overline{x} \rVert$. Then,
\begin{align*}
 & \hspace{0.4cm} h_{y} (t \lVert y - \overline{x} \rVert) \leq h_{y} (\lVert y - \overline{x} \rVert) - \frac{\gamma}{2} t (1-t) \lVert y - \overline{x} \rVert^{2}\\
 & \Longrightarrow ~ h((1-t) \overline{x} + ty) \leq h(y) - \frac{\gamma}{2} t (1-t) \lVert y - \overline{x} \rVert^{2},
\end{align*}
and the result follows.
\end{proof}

Now, we have all the ingredients for showing a function which is strongly star quasiconvex but neither quasiconvex nor quasar-convex.

\begin{example}\label{4:leaf} {\bf (4-Leaf clover example)}
Let $g_{1}: \mathbb{R}_{+} \rightarrow \mathbb{R}$ be defined  by
\[
 g_{1} (x) = \left \{
 \begin{array}[c]{ll}
  3ex^{2}-2ex^{3},~~ & {\rm if} ~ 0\leq x<1,\\
  e^{n} (1 + (e-1) (3(x-n)^{2} - 2(x-n)^{3})), ~~ & {\rm if} ~ n \leq x < n+1, \, \forall \, 
  n \in \mathbb{N}.
 \end{array}
\right.
\]
Then, we define the function $h: \mathbb{R} \rightarrow \mathbb{R}$ by 
$
 h(x) = f(x) + g(x),
$
where $f, g: \mathbb{R} \rightarrow \mathbb{R}$ are given by $f(x)=x^{2}$ 
and
\[
g(x)=\left \{
\begin{array}[c]{ll}%
 g_{1}(-x),~~ & \mathrm{if}~x<0,\\
 g_{1}(x), & \mathrm{if}~x\geq0,
\end{array}
\right.
\]
which is strongly quasiconvex with modulus $\gamma = 2$, but it is not quasar-convex as proved in \cite[Exam\-ple 9]{HLMV}). Furthermore, note that $\min_{\mathbb{R}^{n}}\,h = h(0) =0$. 

Set $\alpha := h(1)$ and $\beta := h\left(  \frac{1}{\sqrt{2}}\right) < \alpha$ (because $h$ is in\-crea\-sing on $\mathbb{R}_{+}$), and let $\lVert \cdot \rVert_{p}$ be the $p$-pseudonorm $\lVert (x_{1}, x_{2}) \rVert_{p} := ( \lvert x_{1} \rvert^{p} + \lvert x_{2} \rvert^{p} )^{\frac{1}{p}}$ where $p := \frac{\ln 2}{\ln(\frac{2 \alpha}{\beta})} < 1$ has been chosen so that $\lVert \left( \frac{1}{2}, \frac{1}{2}\right) \rVert_{p} = \frac{\alpha}{\beta}$. 

Define $\phi: \mathbb{R}^{2}\rightarrow \mathbb{R}$ by (see Figure \ref{fig:exist})
$$
 \phi(x) =\left \{
\begin{array}[c]{ll}
 \frac{\lVert x \rVert}{\lVert x \rVert_{p}} h(\lVert x \rVert), ~~ & \mathrm{if} ~ x \neq 0,\\
 0, & \mathrm{if} ~ x = 0.
\end{array}
\right.
$$
Here, $\mathrm{argmin}_{\mathbb{R}^{n}}\, \phi=\{0\}$ and, for every $y \in \mathbb{R}^{2} \backslash \{0\}$, we have 
\[
\phi_{y}(t) := \phi \left( t \frac{y}{\lVert y \rVert}\right)  =\frac{\lVert y 
\rVert}{\lVert y \rVert_{p}} h(t).
\]
So, $\phi_{y}$ is strongly star quasiconvex with modulus $\frac{\lVert y \rVert
}{\lVert y \rVert_{p}} \gamma = \frac{2 \lVert y \rVert}{\lVert y \rVert_{p}}$. Since $\lVert \cdot \rVert_{p}$ and $\lVert \cdot \rVert$ are equivalent, there exists $m>0$ such that $\frac{\lVert y \rVert}{\lVert y \rVert_{p}} \geq m$ for all $y \neq 0$. Hence, $\phi$ is 
strongly star quasiconvex with modulus $2 m > 0$ by Proposition \ref{char:onedim}.

On the other hand, the restriction of $\phi$ on the $x$ axis is not quasar-convex, i.e., $\phi$ is not quasar-convex (see \cite{HLMV}). Furthermore, $\phi$ is not even quasiconvex. Indeed, let us consider the points $x = (1, 0)$, $y = (0, 1)$ and $z = \frac{1}{2}x + \frac{1}{2}y$. Then, $\phi(x) = \phi(y) = 1 h(1) = \alpha$, but $\phi(z) = \frac{\left \Vert \left(  \frac{1}{2}, \frac{1}{2} \right)  \right \Vert_{p}}{\left \Vert \left(  \frac{1}{2}, \frac{1}{2} \right) \right \Vert} h \left( \left \Vert \left(  \frac{1}{2}, \frac{1}{2} \right) \right \Vert_{p} \right) = \frac{\frac{\alpha}{\beta}}{\frac{1}{\sqrt{2}}} \beta > \alpha$, i.e., $\phi$ is not quasiconvex.

Therefore, $\phi$ is strongly star quasiconvex with modulus $2 m > 0$, without being either quasiconvex or quasar-convex.
\begin{figure}[htbp]
\centering
\includegraphics[scale=0.47]{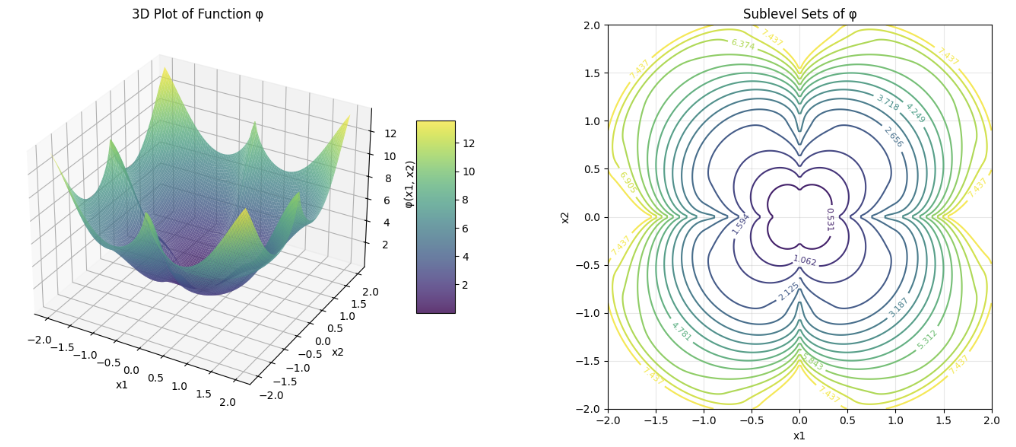}
\caption{An illustration of the function $\phi$ in Example \ref{4:leaf} (left) and its 4-leaf clover sublevel sets (right).} \label{fig:exist}
\end{figure}
\end{example}

\subsection{Star Quasiconvexity}

We begin this subsection with our first main result, which provides a geometric characterization for star quasiconvex functions.

\begin{theorem}\label{geo:charac} {\bf (Star-shaped sublevel sets)}
 Let $h: \mathbb{R}^{n} \rightarrow \overline{ \mathbb{R}}$ be a proper function. Then, $h$ is star quasiconvex with respect to $\overline{x} \in {\rm argmin}_{\mathbb{R}^{n}}\,h$ if and only if $S_{\delta} (h)$ is star-shaped at $\overline{x}$ for all $\delta \in \mathbb{R}$.
\end{theorem}

\begin{proof}
$(\Rightarrow)$ Let $y \in S_{\delta} (h)$. Then, for every $\lambda \in [0,1]$, we define $y_\lambda := \lambda \overline{x} + (1-\lambda) y$. Since $h$ is star quasiconvex, $h(y_\lambda) \le h(y) \le \delta$. Hence, $y_\lambda \in S_{\delta} (h)$, i.e., $[\overline{x}, y] \subset S_{\delta} (h)$. Then, $S_{\delta} (h)$ is star-shaped at $\overline{x} \in {\rm argmin}_{\mathbb{R}^{n}}\,h$.

$(\Leftarrow)$ Assume that $S_{\delta} (h)$ is star-shaped at $\overline{x}$ for all $\delta \in \mathbb{R}$.

First, we show that $\overline{x} \in {\rm argmin}_{\mathbb{R}^n} h$. Suppose for contradiction that there exists $z \in \mathbb{R}^n$ such that $h(z) < h(\overline{x})$. Let $\delta := h(z)$. Then, $z \in S_{\delta} (h)$, and by star-shapedness at $\overline{x}$, the segment $[\overline{x}, z] \subset S_{\delta} (h)$. In particular, $\overline{x} \in S_{\delta} (h)$, so $h(\overline{x}) \le \delta = h(z)$. This contradicts $h(z) < h(\overline{x})$. Hence, $h(\overline{x}) \le h(z)$ for all $z \in \mathbb{R}^n$, so $\overline{x} \in {\rm argmin}_{\mathbb{R}^n}\,h$.

Second, take any $y \in {\rm dom}\,h$ and let $\delta := h(y)$. Then, $y \in S_{\delta}(h)$. By star-shapedness at $\overline{x}$, for every $\lambda \in [0,1]$, the point $y_\lambda := \lambda \overline{x} + (1-\lambda)y \in S_{\delta} (h)$, i.e., $h(y_\lambda) \le \delta = h(y)$. Therefore, $h$ is star quasiconvex with respect to $\overline{x}$.
\end{proof}
%

\begin{proposition} {\rm {\bf (Star-shaped epigraph for star-convexity)}} \label{starshaped:epi}
 Let $h: \mathbb{R}^{n} \rightarrow \overline{ \mathbb{R}}$ be a proper function and $\overline{x} \in {\rm argmin}_{\mathbb{R}^{n} }\,h$. Then, $h$ is star-convex with respect to $\overline{x}$ if and only if ${\rm epi}\,h$ is star-shaped with respect to $(\overline{x}, h(\overline{x}))$.
\end{proposition}

\begin{proof}
($\Rightarrow$) Let $(y, s) \in {\rm epi}\,h$, i.e., $s \ge h(y)$. Since $h$ is star-convex with respect to $\overline{x}$, we have for every $\lambda \in [0,1]$ that
\[
 \lambda h(\overline{x}) + (1-\lambda) s \ge \lambda h(\overline{x}) + (1-\lambda) h(y) \ge h(\lambda \overline{x} + (1-\lambda) y).
\]
Thus, $(\lambda \overline{x} + (1-\lambda) y, \lambda h(\overline{x}) + (1-\lambda) s) \in {\rm epi}\,h$, i.e., ${\rm epi}\,h$ is star-shaped with respect to $(\overline{x}, h(\overline{x}))$.

($\Leftarrow$) Assume that ${\rm epi}\,h$ is star-shaped with respect to $(\overline{x}, h(\overline{x}))$. Then, for any $y \in \mathbb{R}^n$, $(y, h(y)) \in {\rm epi}\,h$. By the star-shapedness, we have $(\lambda \overline{x} + (1-\lambda) y, \lambda h(\overline{x}) + (1-\lambda) h(y)) \in {\rm epi}\,h$ for every $\lambda \in [0,1]$, that is,
\[
\lambda h(\overline{x}) + (1-\lambda) h(y) \ge h(\lambda \overline{x} + (1-\lambda) y).
\]
Thus, $h$ is star-convex with respect to $\overline{x}$.
\end{proof}

\begin{remark}
 \begin{itemize}
  \item[$(a)$] From Proposition \ref{inclusion} and Theorem \ref{geo:charac}, we obtain that the sublevel sets of quasar-convex and star-convex functions are star-shaped at it minimizer. 

  \item[$(b)$] Note that in Proposition \ref{starshaped:epi}, we are assuming that $\overline{x} \in {\rm argmin}_{\mathbb{R}^{n}}\,h$. Indeed, the function $h(x) = x^{2}$ has a convex epigraph, i.e., a star-shaped epigraph at $(1, 1)$, but $x_{0}=1$ is not its minimum.
 \end{itemize} 
\end{remark} 

Theorem \ref{geo:charac} suggests that star quasiconvex functions are quasiconvex on rays starting at the minimizer. Indeed, following \cite{HL-1}, we write (strong) quasiconvexity in the following equivalent manner: For every $x, y \in {\rm dom}\,h$ and $z \in [x,y]$, we have ($\gamma \geq 0$)
\begin{equation} \label{r:new-form}
 h(x) \leq h(y) \, \Longrightarrow \, h(y) \geq h(z) + \frac{\gamma}{2} \lVert z-x \rVert \lVert y-z \rVert.
\end{equation}

\begin{theorem}\label{main:rel}  {\bf (Quasiconvexity along rays)}
 A function $h$ is (strongly) star quasiconvex with modulus $\gamma \geq 0$  with respect to $\overline{x} \in {\rm argmin}_{\mathbb{R}^{n}}\,h$ if and only if its restriction on any half line through $\overline{x}$ is (strongly) qua\-siconvex with the same modulus $\gamma \geq 0$.
\end{theorem}

\begin{proof}
 $(\Leftarrow)$: It follows directly by definition.

 $(\Rightarrow)$: Take any $x, x^{\prime} \in \varepsilon (\overline{x}) := \{\overline{x} + ty: \, t \in \mathbb{R}_{+}\}$ with $y \in {\rm dom}\,h$. We assume without loss of generality that $x \in [\overline{x}, x^{\prime}]$. Take any $t \in [0, 1]$ and set $z := (1-t) x + tx^{\prime}$. Then, $z \in [\overline{x}, x^{\prime}]$. So, applying \eqref{r:new-form}, we have
 \begin{align*}
  h(z) & \leq h(x^{\prime}) - \frac{\gamma}{2} \left \Vert x^{\prime} - z \right \Vert \left \Vert z - \overline{x} \right \Vert \\
  & \leq h(x^{\prime}) - \frac{\gamma}{2} \lVert x^{\prime} - z \rVert \lVert z-x \rVert \\
  & = h(x^{\prime})  - \frac{\gamma}{2} t (1-t) \lVert x^{\prime} - x \rVert^{2}.
 \end{align*}
Hence, $h$ is (strongly) quasiconvex with modulus $\gamma \geq 0$ on $\varepsilon (\overline{x})$ .
\end{proof}

\begin{remark}
 During the preparation of this work, we were pointed out about reference \cite{NTV}, in which the authors studied (strongly) star quasiconvex functions under a different name. However, in virtue of the reasons explained in the introduction and Theorems \ref{geo:charac} and \ref{main:rel}, we propose to keep the {\rm (strongly) star quasiconvex} name for this new notion.

 Furthermore, note that in \cite{NTV} the authors developed the differentiable case and they applied it to constraint optimization problems, while in our contribution, we focus the attention mostly on the nonsmooth case.
\end{remark}

As a consequence of Theorem \ref{main:rel}, we have.

\begin{corollary}\label{nondecreasing:alongrays} {\bf (Nondecreasing along rays)}
Let $K \subseteq \mathbb{R}^{n}$ be a closed set and $h: \mathbb{R}^{n} \rightarrow \overline{\mathbb{R} }$ be a proper function such that $K \subseteq {\rm dom}\,h$. Suppose that $\overline{x} \in {\rm argmin}_{K}\,h$ and $K$ is star-shaped at $\overline{x}$. If $h$ is (strongly) star quasiconvex on $K$ with modulus $\gamma \geq 0$ with respect to $\overline{x}$, then $h$ is nondecreasing along rays through $\overline{x}$ on $K$.
    
Conversely, if $h$ is nondecreasing along rays through $\overline{x}$ on $K$, then $h$ is star quasiconvex with respect to $\overline{x}$ (with $\gamma = 0$).
\end{corollary}

\begin{proof}
Suppose that $h$ is star quasiconvex with respect to $\overline{x}$. Take any ray $\varepsilon (\overline{x}) = \{\overline{x} + t u : t \geq 0\} \cap K$ with $u \in \mathbb{R}^n$. For any $0 \leq t_1 < t_2$, let $\lambda = 1 - \frac{t_1}{t_2} \in [0,1]$. Then, $\overline{x} + t_1 u = \lambda \overline{x} + (1-\lambda)(\overline{x} + t_2 u)$. By star quasiconvexity,
\[
 h(\overline{x} + t_1 u) \leq h(\overline{x} + t_2 u) - \lambda(1-\lambda)\frac{\gamma}{2}\|t_2 u\|^2 \leq h(\overline{x} + t_2 u),
\]
Thus, $h$ is nondecreasing along rays.

Conversely, if $h$ is nondecreasing along rays through $\overline{x}$ on $K$, then for any $y \in K$ and $\lambda \in [0,1]$, the point $\lambda \overline{x} + (1-\lambda)y$ lies on the ray from $\overline{x}$ to $y$. Since $h$ is nondecreasing along this ray, we have
$
h(\lambda \overline{x} + (1-\lambda)y) \leq h(y),
$
i.e., $h$ is star quasiconvex with respect to $\overline{x}$ (with $\gamma = 0$).
\end{proof}

As a consequence, in the case of star quasiconvexity ($\gamma = 0$), we have the following characterization.

\begin{corollary}\label{3;char}
 Let $K \subseteq \mathbb{R}^{n}$ be a closed set and $h: \mathbb{R}^{n} \rightarrow \overline{\mathbb{R}}$ be a proper func\-tion such that $K \subseteq {\rm dom}\,h$. Suppose that $\overline{x} \in {\rm argmin}_{K}\,h$ and $K$ is star-shaped at $\overline{x}$. Then, the following assertions are equivalent:
 \begin{itemize}
  \item[$(a)$] $h$ is star quasiconvex on $K$ with respect to $\overline{x}$;

  \item[$(b)$] $h$ is quasiconvex along rays through $\overline{x}$ on $K$;

  \item[$(c)$] $h$ is nondecreasing along rays through $\overline{x}$ on $K$.
 \end{itemize}
\end{corollary}

Another consequence of Theorem \ref{main:rel} is given below.

\begin{corollary}\label{coro:Charact}  {\bf (Stronger property)}
 Let $K \subseteq \mathbb{R}^{n}$ be a closed set and $h: \mathbb{R}^{n} \rightarrow \overline{\mathbb{R}}$ be a proper func\-tion such that $K \subseteq {\rm dom}\,h$. Suppose that $\overline{x} \in {\rm argmin}_{K}\,h$ and $K$ is star-shaped at $\overline{x}$. If $h$ is (strongly) star quasiconvex on $K$ with modulus $\gamma \geq 0$ with respect to $\overline{x}$, then for every $y \in K$ and every $z = \overline{x} + t(y-\overline{x})$, with $0< t \leq1$, the following holds:
 \begin{equation}\label{no-integral}
  h(z) \leq h(y) - \frac{\gamma}{4} (1-t^{2}) \lVert y - \overline{x} \rVert^{2} \left( = h(y) - \frac{\gamma}{4} (\lVert y - \overline{x} \rVert^{2} - \lVert z - \overline{x} \rVert^{2})  \right). 
 \end{equation}

 \noindent Conversely, if $h$ is continuous and \eqref{no-integral} holds for every $y \in K$, then $h$ is (strongly) star quasiconvex on $K$ with the same modulus $\gamma \geq 0$ with respect to $\overline{x}$.
\end{corollary}

\begin{proof}
 Just take $x = \overline{x}$ in the proof of \cite[Theorem 4]{HL-1}.
\end{proof}

From Corollary \ref{coro:Charact}, we can obtain a quadratic growth property for the unique minimizer of strongly star quasiconvex function.

\begin{proposition}\label{prop:qwc}  {\bf (Quadratic growth property)}
 Let $K \subseteq \mathbb{R}^{n}$ be a closed set and $h: \mathbb{R}^{n} \rightarrow \overline{\mathbb{R}}$ be a proper func\-tion such that $K \subseteq {\rm dom}\,h$. Suppose that $\overline{x} \in {\rm argmin}_{K}\,h$ and $K$ is star-shaped at $\overline{x}$. If $h$ is strongly star quasiconvex on $K$ with modulus 
$\gamma > 0$, then
 \begin{equation}\label{qgp:classb}
  h(\overline{x}) + \frac{\gamma}{4} \lVert y - \overline{x} \rVert^{2} \leq h(y), ~ \forall ~ y \in K,
 \end{equation}
\end{proposition}

\begin{proof}
 Exactly as the proof of \cite[Corollary 9]{HL-1}.
\end{proof}

The following result provides information regarding the behaviour of strongly star quasiconvex functions at infinity, a result which extends and improves \cite[Theo\-rem 1]{Lara-9} and \cite[Proposition 12]{BLL} to strong star quasiconvexity and star-shaped sets.

\begin{proposition}\label{spconvex:super} {\bf ($2$-Supercoercivity on star-shaped sets)}
 Let $K \subseteq \mathbb{R}^{n}$ be a closed set and $h: \mathbb{R}^{n} \rightarrow \overline{\mathbb{R}}$ be a proper func\-tion such that $K \subseteq {\rm dom}\,h$. Suppose that $\overline{x} \in {\rm argmin}_{K}\,h$ and $K$ is star-shaped at $\overline{x}$. If $h$ is strongly star quasiconvex function on $K$ with modulus $\gamma > 0$, then $h$ is $2$-supercoercive.
\end{proposition}

\begin{proof}
 Let $\{x_{k}\}_{k} \subseteq K$ be such that $\lVert x_{k} \rVert \rightarrow + \infty$. Then, we may assume that $\frac{x_{k}}{\lVert x_{k} \rVert} \rightarrow u \in K^{\infty}$ with $\lVert u \rVert= 1$. Since $\overline{x} \in {\rm argmin}_K\,h$, $h(\overline{x}) \leq h(x_{k})$ for all $k \in \mathbb{N}$. Using relation \eqref{qgp:classb}, we have the chain of implications
\begin{align*} 
 h(x^{k}) \geq h(\overline{x}) + \frac{\gamma}{4} \lVert x_{k} - \overline{x} \rVert^{2} & \Longrightarrow \frac{h(x_{k})}{\lVert x_{k} \rVert^{2}} \geq \frac{h(\overline{x})}{\lVert x_{k} \rVert^{2}} + \frac{\gamma}{4} \left \Vert \frac{x_{k}}{\lVert x_{k} \rVert} - \frac{\overline{x}}{\lVert x_{k} \rVert} \right \Vert^{2} \\
 & \Longrightarrow \liminf_{k \rightarrow+ \infty} \frac{h(x_{k})}{\lVert x_{k} \rVert^{2}} \geq \frac{\gamma}{4} > 0,
\end{align*}
 i.e., $h$ is 2-super\-coercive.
\end{proof}

Now, we focus our attention on local/global minimum properties for (strongly) star quasiconvex functions. 

As a first result, note that for strongly star quasiconvex functions every local minimizer is global.

\begin{proposition}\label{loc:global}
 Let $h: \mathbb R^n \to \overline{\mathbb{R}}$ be a proper function and $\overline{x} \in {\rm argmin}_{\mathbb{R}^{n}}\,h$. Assume that $h$ is strongly star quasiconvex. Then the following assertions hold:
 \begin{enumerate}
  \item[$(a)$] if $x$ is a local minimizer of $h$, then $x \in {\rm argmin}_{\mathbb{R}^{n}}\,h$;

  \item[$(b)$] every local maximizer of $h$ belongs to ${\rm argmin}_{\mathbb{R}^n}\,h$. In particular, \(h\) has no strict local maximizers.
 \end{enumerate}
\end{proposition}

\begin{proof}
$(a)$: Suppose for the contrary that $x_{0}$ is a local minimizer of $h$ and $x_{0} \notin {\rm argmin}_{\mathbb{R}^{n}}\,h$. Then, $h(x_{0}) > h(\overline{x})$ and $\overline{x} \neq x_{0}$. For $\lambda \in \, ]0,1]$, set $x_\lambda := \lambda \overline{x} + (1-\lambda) x_{0}$. By strong star quasiconvexity, it is evident that
\[
 h(x_\lambda) \le h(x_{0}) - \lambda (1-\lambda) \frac{\gamma}{2} \lVert \overline{x} - x_{0} \rVert^{2} < h(x_{0}).
\]
Since $x_\lambda \to x_{0}$ as $\lambda \downarrow 0$, this contradicts the local minimality of $x_{0}$. Therefore, $x_{0} \in {\rm argmin}_{\mathbb{R}^{n}}\,h$.
\\
$(b)$: Suppose that $x_{0} \notin {\rm argmin}_{\mathbb{R}^{n}}\,h$. Then, $h(x_{0}) > h(\overline x)$ and $\overline{x} \neq x_{0}$. For $\varepsilon > 0$, let us define $y_\varepsilon := \overline{x} + (1 + \varepsilon) (x_{0} - \overline x)$ and $\lambda_\varepsilon := \frac{\varepsilon}{1 + \varepsilon} \in \, ]0, 1[$. Then, $x_{0} = \lambda_\varepsilon \overline x + (1-\lambda_\varepsilon) y_\varepsilon$ and $ y_\varepsilon \to x_{0}$ as $\varepsilon \downarrow 0$. Since $h$ is strongly star quasiconvex at $\overline x$, we get
$$h(x_{0}) = h(\lambda_\varepsilon \overline x + (1-\lambda_\varepsilon) y_\varepsilon) \le h(y_\varepsilon) - \lambda (1-\lambda) \frac{\gamma}{2} \lVert \overline{x} - x_{0} \rVert^{2} < h(y_{\varepsilon}).$$ 
Hence, every neighborhood of $x_{0}$ contains a point $y_\varepsilon$ such that $h(y_\varepsilon) > h(x_{0})$, proving that $x_{0}$ is not a local maximizer. Therefore, every local maximizer of $h$ belongs to ${\rm argmin}_{\mathbb{R}^n}\,h$. 

Finally, if $x_{0}$ were a strict local maximizer, then all nearby points would satisfy $h(y) < h(x_{0})$, which contradicts the fact that $x_{0} \in {\rm argmin}_{\mathbb{R}^{n}}\,h$, i.e., $h$ accepts no strict local maximizers.
\end{proof}

\begin{remark}
 For star quasiconvex functions ($\gamma=0$), one may find local minimizers which are not global. Indeed, since quasiconvex functions are star quasiconvex, the function $h(x) = \min\{\lvert x \rvert, 1\}$ is star quasiconvex with local minimizers which are not global.
\end{remark} 

The following result is useful for constructing (strongly) star quasiconvex functions which are neither convex nor quasiconvex (see \cite{HADR, HSS} for the quasar-convex case). Note that in the next proposition, differentiability of $g$ is not needed.

\begin{proposition}\label{prop2:b}
 Let $f: \mathbb{R} \rightarrow \mathbb{R}$ be a function, $0 \in {\rm argmin}_{\mathbb{R}}\,f$, $f(0)=0$, and $g: \mathbb{R}^{n} \rightarrow \mathbb{R}$ be a function such that $g \geq 1$. Define the function
 \begin{equation}\label{exam:const}
  h(x) := f(\lVert x \rVert) g \left( \frac{x}{\lVert x \rVert} \right), ~ \forall ~ x \in \mathbb{R}^{n} \backslash \{0\},
 \end{equation}
 and $h(0) := 0$. If $f$ is (strongly) star quasiconvex with modulus $\gamma \geq 0$, then $h$ is (strongly) star quasiconvex with modulus $\gamma \geq 0$.
\end{proposition}

\begin{proof}
 Let $y \in \mathbb{R}^{n}$, $\overline{x} = 0 \in {\rm argmin}_{\mathbb{R}}\,f$, and $f(\overline{x})=0$.  We have 
 $$h(\lambda \overline{x} + (1-\lambda) y) = h((1-\lambda) y) = f((1 - \lambda) \lVert y \rVert) g \left( \frac{y}{\lVert y \rVert} \right), ~ \forall 
 ~ \lambda \in [0, 1].$$
 Since $f$ is strongly star quasiconvex with modulus $\gamma \geq 0$, we have
 \begin{align*}
  f((1-\lambda) \lVert y \rVert) g \left( \frac{y}{\lVert y \rVert} \right) & 
  \leq \left( f(\lVert y \rVert) - \lambda (1-\lambda) \frac{\gamma}{2} \lVert y \rVert^{2} \right) g \left( \frac{y}{\lVert y \rVert} \right) \\
  & = h(y) - \lambda (1-\lambda) \frac{\gamma}{2} \lVert y \rVert^{2} g 
  \left( \frac{y}{\lVert y \rVert} \right).
 \end{align*}
 Since $g \geq 1$, we have
 \begin{align*}
  h((1-\lambda) y) + \lambda (1-\lambda) \frac{\gamma}{2} \lVert y \rVert^{2} 
  & \leq h(y) + \lambda (1-\lambda) \frac{\gamma}{2} \lVert y \rVert^{2} \left( 1 - g \left( \frac{y}{\lVert y \rVert} \right) \right) \\
  & \leq h(y), ~ \forall ~ \lambda \in [0, 1],
 \end{align*}
 i.e., $h$ is strongly star quasiconvex with modulus $\gamma \geq 0$.
\end{proof}

The following example shows the nonconvex behaviour of a strongly star quasiconvex function of type \eqref{exam:const}.

\begin{example}\label{exam:02} 
 Let $\alpha \in \, ]0, 1[$, $k \in \mathbb{N}$, $t,x_1,x_2\in\mathbb R$,
 \begin{align*}
  & \hspace{0.7cm} f(t) = \max\{\lvert t \rvert^{\alpha}, \lvert t \rvert^{2} - k\}, ~~ {\rm and} \\
  & g (x_1, x_2)=\frac{1}{4N} \sum_{i=1}^{N} \left( a_{i} \sin(b_{i} x_1)^{2}+c_{i} \cos(d_{i} x_2)^{2} \right),
 \end{align*} 
 with $N=10$, ${a_{i}}$ and ${c_{i}}$ are independently and uniformly distributed on $[0, 20]$, while ${b_{i}}$ and ${d_{i}}$ are independently and uniformly distributed on $[-25, 25]$.

Then, $f$ is strongly quasiconvex on $\mathbb{R}^{n}$ as the maximum of two strongly quasiconvex functions (see \cite{GLM-Survey}), thus strongly star quasiconvex by Proposition \ref{inclusion}, while $g$ satisfies the assumptions of Proposition \ref{prop2:b}. Therefore, 
\begin{align*}
 h(x_{1}, x_{2}) & = f( \lVert (x_{1}, x_{2}) \rVert) g \left( \frac{(x_{1}, x_{2})}{ \lVert (x_{1}, x_{2}) \rVert}\right) 
\end{align*}
is strongly star quasiconvex. Note that $h$ is nonsmooth because $f$ is nonsmooth. We plot function $h$ in Figures \ref{fig:02} and \ref{fig:03}, below.
\begin{figure}[htbp] 
\centering \includegraphics[width=1.00\linewidth]{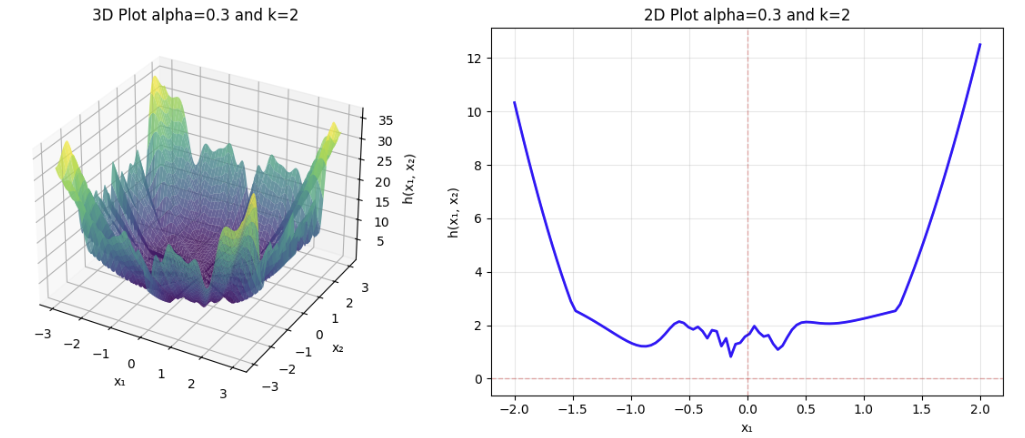} 
\caption{Function $h$ in Example \ref{exam:02} when $\alpha = 0.3$ and $k=2$. 
A 3D plot of $h$ (left) and an arbitrary segment that does not contain the minimizer (right).} \label{fig:02}
\end{figure}
\begin{figure}[htbp] 
\centering \includegraphics[width=1.00\linewidth]{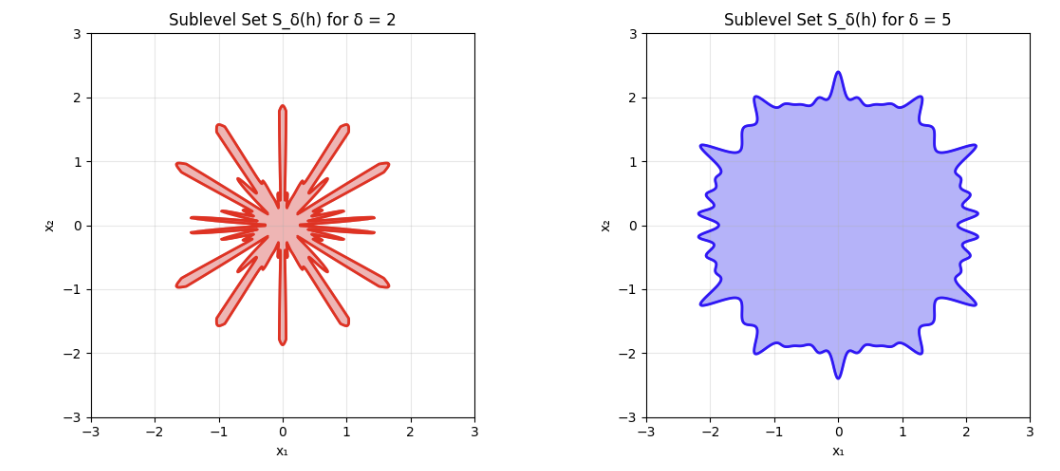} 
\caption{An illustration of the sublevel sets at height $\delta = 2$ (left) and $\delta = 5$ (right) of function $h$ des\-cri\-bed in Example \ref{exam:02} when $\alpha = 0.3$ and $k=2$.} \label{fig:03}
\end{figure}
\end{example}

Another interesting sufficient condition for functions to be star quasiconvex ($\gamma=0$) is given below for positively homogenous functions of any positive degree.

\begin{proposition}\label{pos:hom}
 Let $h: \mathbb{R}^{n} \rightarrow \mathbb{R}$ and $\overline{x} = 0 \in {\rm argmin}_{\mathbb{R}^{n}}\,h$. If $h$ is positively homogeneous of degree $p> 0$, then $h$ is star quasiconvex with respect to $\overline{x} = 0$.
\end{proposition}

\begin{proof}
 Since $h(y) \geq 0$ for all $y \in \mathbb{R}$ and $\lambda \in [0, 1]$, we have $(\lambda^{p} - 1) \leq 0$ for all $p>0$. Then, 
\begin{align*}
 h(y) \geq 0 & \Longleftrightarrow ~ (\lambda^{p} - 1) h(y) \leq 0 \notag \\
 & \Longleftrightarrow ~ \lambda^{p} h(y) \leq h(y)  \notag \\
 & \Longrightarrow ~ h(\lambda y) \leq h(y), ~ \forall ~ \lambda \in [0, 1] \notag \\
 & \Longleftrightarrow ~ h( (1-t) y) \leq h(y), ~ \forall ~ t \in [0, 1],
\end{align*}
i.e., $h$ is star quasiconvex at $0$.
\end{proof}

Now, let us study the properties of the proximity operator of a strongly star quasiconvex function $h$. To that end, we first observe:

\begin{remark}
 Let $\beta > 0$ in the statements below. 
\begin{itemize}
 \item[$(i)$] If $h$ is lsc and strongly star quasiconvex with $\gamma > 0$, then   ${\rm Prox}_{\beta h} \neq \emptyset$ and compact because $h$ is $2$-supercoercive by Corollary \ref{spconvex:super}.

 \item[$(ii)$] Since strongly quasiconvex functions are strongly star quasiconvex, ${\rm Prox}_{\beta h}$ is not necessarily a singleton in virtue of \cite[Remark 6]{Lara-9}.

 \item[$(iii)$] If $h$ is just star quasiconvex ($\gamma = 0$), then ${\rm Prox}_{\beta h}$ could be empty as for the quasiconvex function $h(x) = x^{3}$.
\end{itemize}
\end{remark}

A connection between the fixed points of the proximity operator and the unique minimizer of a strongly star quasiconvex function is given below. We emphasize that in the next two results, the considered set $K$ is star-shaped (but not necessarily convex).

\begin{proposition}\label{prop:fixed} {\bf (Fixed points criteria on star-shaped sets)}
 Let $K \subseteq \mathbb{R}^{n}$ be a closed set, $h: \mathbb{R}^{n} \rightarrow \overline{\mathbb{R}}$ be a proper lsc function such that $K \subseteq {\rm dom}\,h$, and $\beta > 0$. Suppose that $\overline{x} \in {\rm argmin}_{K}\,h$ and $K$ is star-shaped at $\overline{x}$. If $h$ is strongly star quasiconvex function on $K$ with modulus $\gamma > 0$, then
 \begin{equation*}\label{fixed:points}
  {\rm Fix} \left( {\rm Prox}_{\beta h} (K, \cdot) \right) = \frac{\gamma}{4}{\rm -argmin}_{K}\,h.
 \end{equation*}
\end{proposition}

\begin{proof}
 $(\supseteq)$ Straightforward. 
 
 $(\subseteq)$ Let $x^{*} \in {\rm Prox}_{\beta h} (K, x^{*})$. Then, $h(x^{*}) \leq h(x) + \frac{1}{2 \beta} \lVert x^{*} - x \rVert^{2}$ for all $x \in K$. Since $K$ is star-shaped at $\overline{x}$, take $x = \lambda \overline{x} + (1-\lambda) x^{*} \in K$ with $\lambda \in [0, 1]$. Since $h$ is strongly star quasiconvex on $K$ with modulus $\gamma > 0$, we have

 $h(x^{*}) \leq h(\lambda \overline{x} + (1-\lambda) x^{*}) + \frac{1}{2 \beta} \lVert \lambda (\overline{x} - x^{*} ) \rVert^{2}$
 
 $h(x^{*}) - \lambda \left( 1 - \lambda \right) \frac{\gamma}{2} \lVert x^{*} - \overline{x} \rVert^{2} + \frac{\lambda^{2}}{2 \beta} \lVert x^{*} - \overline{x} \rVert^{2},$
 
and so 

$\frac{\lambda}{2} \left( \gamma - \lambda \gamma - \frac{\lambda}{\beta} \right) \lVert x^{*} - \overline{x} \rVert^{2} \leq 0, ~ \forall ~ \lambda \in \, ]0, 1].$

 Taking $0 < \lambda < \frac{\gamma \beta}{1 + \gamma \beta} < 1$, we obtain $\lVert x^{*} - \overline{x} \rVert^{2} \leq 0$, thus $x^{*} = \overline{x} \in {\rm argmin}_{K}\,h$.
 
 Since $h$ is strongly star quasiconvex with modulus $\gamma > 0$, the minimizer $\overline{x}$ is unique and satisfies the quadratic growth property \eqref{qgp:classb}, i.e., $x^{*} = \overline{x} \in \frac{\gamma}{4}{\rm -argmin}_{K}\,h$, which com\-ple\-tes the proof. 
\end{proof}

A relationship between the minimizer of the proximity operator and the minimizers of the function is given below

\begin{proposition}\label{prox:sqcx} {\bf (Relationship with proximity operator on star-sha\-ped sets)}
 Let $K \subseteq \mathbb{R}^{n}$ be a closed set, $h: \mathbb{R}^{n} \rightarrow \overline{\mathbb{R}}$ be a proper lsc function such that $K \subseteq {\rm dom}\,h$, $\beta > 0$, and $z \in K$. Suppose that $\overline{x} \in {\rm argmin}_{K}\,h$, $K$ is star-shaped at $\overline{x}$ and $h$ is strongly star quasiconvex function on $K$ with modulus $\gamma > 0$. If $x^{*} \in {\rm Prox}_{\beta h} (K, z)$, then
 \begin{align}\label{lin:conv1}
  \frac{\gamma}{2} \lVert x^{*} - \overline{x} \rVert^{2} \leq \frac{1}{\beta} \langle x^{*} - z, \overline{x} - x^{*} \rangle.
 \end{align}
\end{proposition}

\begin{proof}
 Since $x^{*} \in {\rm Prox}_{\beta h} (K, z)$, we have
 \begin{align*}\label{ppa:dec}
  h(x^{*}) + \frac{1}{2 \beta} \lVert x^{*} - z \rVert^{2} & \leq h(x) + \frac{1}{2 \beta} \lVert x - z \rVert^{2}, ~ \forall ~ x \in K. 
 \end{align*} 
 Since $K$ is star-shaped at $\overline{x}$ and $h$ is strongly star quasiconvex on $K$ with modulus $\gamma > 0$, taking $x = \lambda \overline{x} + (1 - \lambda) x^{*} \in K$ for all $\lambda \in [0, 1]$, we have
 \begin{align}
  h(x^{*}) + & \frac{1}{2 \beta} \lVert x^{*} - z \rVert^{2} \leq h(\lambda \overline{x} + (1 - \lambda) x^{*}) + \frac{1}{2 \beta} \lVert \lambda 
  (\overline{x} - z) + (1-\lambda)(x^{*} - z) \rVert^{2}, \notag \\
  & \leq h(x^{*}) - \lambda (1-\lambda) \frac{\gamma}{2} \lVert x^{*} - \overline{x} \rVert^{2} + \frac{\lambda}{2 \beta} \lVert z - \overline{x} \rVert^{2} + \frac{(1-\lambda)}{2 \beta} \lVert x^{*} - z \rVert^{2} \notag \\
  & ~~~  - \frac{\lambda (1-\lambda)}{2 \beta} \lVert x^{*} - \overline{x} \rVert^{2}, ~ \forall ~ \lambda \in [0, 1], \notag
\end{align} 
in virtue of \eqref{iden:1}. Then, it follows from \eqref{3:points} that
 \begin{align*}
  & ~\, 0 \leq \frac{\lambda}{\beta} \langle x^{*} - z, \overline{x} - x^{*} \rangle + \frac{\lambda}{2} \left( \frac{\lambda}{\beta} - \gamma + \lambda \gamma \right) \lVert x^{*} - \overline{x} \rVert^{2},
 \end{align*}
 and the result follows by multiplying by $\frac{1}{\lambda}$, and then taking $\lambda \downarrow 0$.
\end{proof}

\subsection{The Differentiable Case}

Under differentiability, more can be said regarding (strongly) star quasiconvex functions. As a first result, we revisit the following characterization (see \cite[Theorem 3.3]{NTV}).

\begin{proposition}\label{char:diff} {\bf (First-order characterization)}
 Let $h: \mathbb{R}^{n} \rightarrow \mathbb{R}$ be a di\-ffe\-ren\-tia\-ble function and $\overline{x} \in {\rm argmin}_{\mathbb{R}^{n}}\,h$. Then, $h$ is (strongly) star qua\-si\-con\-vex with modulus $\gamma \geq 0$ with respect to $\overline{x}$ if and only if for every $y \in \mathbb{R}^{n}$, we have
 \begin{equation}\label{char:grad}
  \langle \nabla h(y), y - \overline{x} \rangle \geq \frac{\gamma}{2} \lVert y - \overline{x} \rVert^{2}.
 \end{equation}
\end{proposition}

\begin{proof}
 $(\Rightarrow)$: Exactly as in \cite[Theorem 3.3]{NTV}.

$(\Leftarrow)$: Let $y \in \mathbb{R}^{n}$ and $z := \overline{x} + t(y -
 \overline{x})$, with $0<t<1$. Take any $s \in [t, 1]$ and set $\overline{x}_{s} 
 := \overline{x} + s(y-\overline{x})$, $g(s) = h(\overline{x}_{s})$. Then, $\left \langle \nabla h(\overline{x}_{s}), \overline{x}_{s} - 
 \overline{x} \right \rangle \geq \frac{\gamma}{2} \lVert \overline{x}_{s} -
 \overline{x} \rVert^{2}$. Thus, $\left \langle \nabla h(\overline{x}_{s}), y -
 \overline{x} \right \rangle \geq \frac{\gamma}{2} s \lVert y - \overline{x} 
 \rVert^{2}$ or $g^{\prime} (s) \geq \frac{\gamma}{2}s \lVert y - \overline{x}
 \rVert^{2}$. By integrating between $t$ and $1$, we have
 \begin{align*}
  & \hspace{2.7cm} g(1) - g(t) \geq \frac{\gamma}{4} \lVert y - \overline{x}
  \rVert^{2} (1-t^{2})  \\
  & \Longleftrightarrow ~ h(y) - h (\overline{x} + t(y - \overline{x})) \geq
  \frac{\gamma}{4} \lVert y-x \rVert^{2} (1-t^{2}).
 \end{align*}
 Hence, $h$ satisfies \eqref{no-integral}, i.e., $h$ is (strongly) star quasiconvex with mo\-du\-lus $\gamma \geq 0$ with respect to $\overline{x}$ by Corollary \ref{coro:Charact}.
\end{proof}

\begin{remark}
 \begin{itemize}
  \item[$(i)$] Observe that when $\gamma > 0$, the previous proposition shows that differentiable strongly star quasiconvex functions coincides with functions satisfying the {\it Restricted Secant Inequality} (RSI) property (see \cite{Yi,ZY} among others). 
 
  \item[$(ii)$] When $\gamma = 0$, differentiable star quasiconvex functions includes {\it Variationally Coherent} (VC) functions, which were developed in \cite{ZBoyd} (or differentiable functions which satisfies \cite[Assumption $(1.2)$]{SS}). This inclusion may be strict as the following example shows.

  Recall that a differentiable function $\phi$ is said to be Variationally Coherent if:
  \begin{equation}\label{VC}
  \langle \nabla \phi(x), x - x^{*} \rangle \geq 0, ~ \forall ~ x \in \mathbb{R}^{n}, ~ \forall ~ x^{*} \in {\rm argmin}_{\mathbb{R}^{n}}\,\phi, \tag{VC}
  \end{equation}
  and there exists some $x^{*}_{0} \in {\rm argmin}_{\mathbb{R}^{n}}\,\phi$ such that equality holds only if $x \in {\rm argmin}_{ \mathbb{R}^{n}}\,\phi$.
 
  Let $h: \mathbb{R}^{2} \rightarrow \mathbb{R}$ be given by $h(x, y) = x^{2} (1+y^{2})$. Here ${\rm argmin}_{\mathbb{R}^{2}}\,h = \{(0, t): t \in \mathbb{R}\}$ and $\min \,h = 0$. By Proposition \ref{pos:hom}, $h$ is star quasiconvex with respect to $(0, 0)$. However, it is not VC since relation \eqref{VC} does not hold for every minimizer as we show next. Indeed, if it is VC, then 
\begin{align*}
 \langle \nabla h(x, y), (x,y) - (0, t) \rangle & = 2x^{2}(1 + 2y^{2} -ty) \geq 0, ~ \forall ~ x, y, t \in \mathbb{R}.
\end{align*} 
 Take for instance $x=1$, $y=\frac{1}{4}$ and $t=8$. Then,
 $$2(1)^{2}(1 + 2 (\frac{1}{4})^{2} - 8 \frac{1}{4}) = 2 (1 + \frac{1}{8} - 2) < 0.$$
 Therefore, $h$ is not VC. 
 \end{itemize}
\end{remark}

Recall that a differentiable function $h: \mathbb{R}^{n} \rightarrow \mathbb{R}$ satisfies the PL property \cite{L,P1} if there exists $\mu > 0$ such that
\begin{equation*}\label{PL:def}
 \lVert \nabla h(x) \rVert^{2} \geq \mu (h(x) - h(\overline{x})), ~ \forall ~ x \in \mathbb{R}^{n},
\end{equation*}
where $\overline{x} \in {\rm argmin}_{ \mathbb{R}^{n}}\,h$. The PL property is implied by strong convexity, but it allows for multiple minimizers and does not require any convexity assumption (see \cite{Ka-Nu-Sc}).

Furthermore, as noted in \cite[Lemma 4]{HADR}, $L$-smooth functions which satisfies relation \eqref{char:grad} with $\gamma > 0$ are strongly quasar-convex, i.e., we have the following sufficient condition for differentiable strongly star quasiconvex functions to be strongly quasar-convex.

\begin{corollary} \label{imply:quasar}
 Let $h: \mathbb{R}^{n} \rightarrow \mathbb{R}$ be a differentiable with $L$-Lipschitz continuous gradient ($L>0$) and strongly star quasiconvex with modulus $\gamma > 0$. If $0 < \rho < \min\{ \frac{2 \gamma}{L}, 1\}$, then $h$ is $(\rho, \gamma^{\prime})$-strongly quasar-convex with modulus $\rho \in \, ]0, 1]$ and $\gamma^{\prime} = \frac{\gamma}{\rho} - \frac{L}{2} > 0$.
\end{corollary}

Another consequence of Proposition \ref{char:diff} is that strongly star quasiconvex functions satisfy the PL property when they have a Lipschitz continuous gradient. Since the proof is similar to the one in \cite[Proposition 7]{LMV}, it is omitted.

\begin{proposition}\label{prop:PL}  {\bf (Polyak-{\L}ojasiewicz property)}
 Let $h: \mathbb{R}^{n} \rightarrow \mathbb{R}$ be a differentiable with $L$-Lipschitz continuous gradient ($L>0$) that is strongly star quasiconvex with modulus $\gamma > 0$. Then, the PL property holds with modulus $\mu := \frac{\gamma^{2}}{2L} > 0$, that is,
 \begin{equation*}\label{eq:PL}
  \lVert \nabla h(x) \rVert^{2} \geq \frac{\gamma^{2}}{2L} (h(x) - h(\overline{x})), ~ \forall ~ x \in  \mathbb{R}^n,
 \end{equation*}
 where $\overline{x} = {\rm argmin}_{\mathbb{R}^n}\,h$.
\end{proposition}


\section{Linear Convergence of Gradient-type Me\-thods Revisited}\label{sec:04}

Using Propositions \ref{char:diff} and \ref{prop:PL}, and Corollary \ref{imply:quasar}, we can ensure the linear convergence of the Heavy-ball and Nesterov accelerations of the gradient method. These results were already proved in, for instance, \cite{HLMV,HADR,LMV}. We mention them here for completeness. 

We emphasize that during the whole section, the function $h$ is assumed to be strongly star quasiconvex with modulus $\gamma > 0$.

\begin{theorem}\label{convergenceHeavyball} {\bf (Linear convergence of heavy-ball acceleration)}
 Let $h: \mathbb{R}^n \rightarrow \mathbb{R}$ be a differentiable strongly star quasiconvex func\-tion with modulus $\gamma > 0$ with $L$-Lipschitz continuous gradient and $\{\overline{x}\} = {\rm argmin}_{\mathbb{R}^n} h$. Suppose that
 \begin{equation}
  \alpha \in \,  \left[0, \frac{\sqrt{2}}{2}\right[, ~~~ \beta  \in \, \left]0, \frac{1 - 2 \alpha^{2}}{L} \right]. \notag
 \end{equation} 
 Then, the se\-quen\-ce $\{x_k\}_{k}$, generated by the heavy ball method:
 \begin{align*}
\left\{
 \begin{array}{ll}\label{inertialHess:alg}
  ~~~ y_{k} = x_{k} +   \alpha (x_{k} - x_{k - 1}), \\ [2mm]
  x_{k+1} = y_{k} - \beta \nabla h(x_{k}),
 \end{array}
 \right.
\end{align*}
converges linearly to $\overline{x}$ and the sequence $\{h(x_k)\}_{k}$ converges linearly to the optimal value $h^{*} := h(\overline{x})$.
\end{theorem}

\begin{proof}
 The same as that in \cite[Theorem 13]{HLMV} with $\theta = 0$.
\end{proof}

As a direct consequence, we have the following convergence rate result.

\begin{corollary}{\bf (Convergence rate)} \label{HB:rates}
 Under the assumptions of Theorem \ref{convergenceHeavyball}, we have
 the following convergence rate
 \begin{align*}
  & \frac{\gamma}{4} \|x_{k+1} - \overline{x}\|^2 \leq h(x_{k+1}) - h^{*} \leq \left(1 - \frac{\rho}{\sigma} \right)^k \mathcal{E}_{1}, \\
  & ~ \|x_{k+1} - x_{k}\|^2 \leq \frac{\beta}{\alpha^2} \left(1-\frac{\rho}{\sigma} \right)^k \mathcal{E}_1,
 \end{align*}
 where $\mathcal{E}_{1} := := h(x_{k}) - h^* + \frac{\alpha^2}{\beta} \lVert x_{k} - x_{k-1} \rVert^2$, $\rho := \min\left\{ \frac{\beta}{2}, \frac{1}{2 \beta} \left(1 - \beta L - 2 \alpha^2 \right) \right\}$, and $\sigma := \max\left\{ \frac{15}{\beta}, \frac{2L}{\gamma^2} + 15 \beta \right\}$.
\end{corollary}

\begin{remark}
 Note that if $\alpha = 0$, then we obtain convergence rates for the classical gradient method.
\end{remark}

We also recall the linear convergence rate for the Nesterov acceleration obtained in \cite{HADR}.

\begin{theorem}\label{convergenceNesterov2} {\bf (Linear convergence of Nesterov acceleration)} {\rm (\cite[Coro\-lla\-ry 1]{HADR})}
 Let $h: \mathbb{R}^n \rightarrow \mathbb{R}$ be a differentiable strongly star quasiconvex func\-tion with modulus $\gamma > 0$ with $L$-Lipschitz continuous gradient and $\overline{x} \in{\rm argmin}_{\mathbb{R}^n} h$. Take $0 < \rho < \min\{ \frac{2 \gamma}{L}, 1\}$. Let $(x_k)_{k \in \mathbb{N}}$ be generated by: 
 \begin{align*}
\left\{
 \begin{array}{ll}
   y_{k} = \alpha_{k} x_{n} + (1 - \alpha_{k}) z_{k}, \\ [2mm]
   x_{k+1} = y_{k} - s \nabla h(y_{k}), \\ [2mm]
   z_{k+1} = \beta_{k} z_{k} + (1 - \beta_{k}) y_{k} - \eta_{k} \nabla h(y_{k}),
 \end{array}
 \right.
\end{align*}
with $z_{0} = x_{0}$ and parameters
\[
 s = \rho^2 \frac{\gamma}{L^2}, \quad \alpha_k = \frac{1}{1 + \sqrt{\gamma s}}, \quad \beta_k = 1 - \rho \sqrt{\gamma s}, \quad \eta_k = \frac{\sqrt{s}}{\sqrt{\gamma}}.
\]
Then,
\begin{equation}\label{lin:nesterov}
 h(x_k) - h^* \le \frac{2}{\rho} \left(1 - \rho^2 \frac{\gamma}{L} \right)^{k} (h(x_0) - h^*), ~ \forall ~ k \in \mathbb{N}.
\end{equation}
\end{theorem}

As a direct consequence of the above theorem, we have the following convergence rate result.

\begin{corollary}{\bf (Convergence rate)} \label{N:rates}
 Under the assumptions of Theorem \ref{convergenceNesterov2}, we have the following convergence rate
 \begin{align*}
 \|x_k - x^*\| \le \sqrt{\frac{4 (2 - \rho)}{\rho^2 \gamma} \left(1 - \rho^2 \frac{\gamma}{L} \right)^k (h(x_0) - h^*)}.
 \end{align*}
\end{corollary}

\begin{proof}
 We simply use the quadratic growth property for strongly quasar-convex functions (see \cite[Corollary 1]{HSS}), which says that $ \frac{\rho \gamma}{2 (2 - \rho)} \lVert x - \overline{x} \rVert^{2} \leq h(x) - h^{*}$ in relation \eqref{lin:nesterov}.
\end{proof}

%

\section{Linear Convergence of PPA on Star-Shaped Sets}\label{sec:05}

As in the previous section, we emphasize that the function $h$ is assumed to be strongly star quasiconvex with modulus $\gamma > 0$.

Let us consider the classical version of the proximal point algorithm.

\begin{algorithm}[H]
\caption{PPA for Strongly Star Quasiconvex Functions (PPA-SSQC)}\label{ppa:sqcx}
\begin{description}
 \item[Step 0.]  Let $K \subseteq \mathbb{R}^{n}$ be a closed and star-shaped set at the unique minimizer of $h$ on $K$. Take $x^{0} \in K$, $k=0$, and  a sequence of positive numbers $\{\beta_{k}\}_{k \in \mathbb{N}_{0}}$. 

 \item[Step 1.] Take
 \begin{equation}\label{step:sqcx}
  x^{k+1} \in {\rm Prox}_{\beta_{k} h} (K, x^{k}).
 \end{equation}

 \item[Step 2.] If $x^{k+1} = x^{k}$, then STOP, $\{x^{k}\} = {\rm argmin}_{K}\,h$. Otherwise, take $k=k+1$ and go to Step 1.
 \end{description}
\end{algorithm}

Note that the iterates exist because $h$ is strongly star quasiconvex (thus $2$-supercoercive), lsc, and the stopping criterion holds by Proposition \ref{prop:fixed}.

\medskip

In the next result, we ensure linear convergence for the PPA for strongly star
quasiconvex functions ($\gamma > 0$) on a closed and star-shaped set $K$ at the minimizer.

\begin{theorem}\label{conver:sqcx} {\bf (Linear convergence of the PPA on star-shaped domains)}
 Let $K \subseteq \mathbb{R}^{n}$ be a closed set and $h: \mathbb{R}^{n} \rightarrow \overline{\mathbb{R}}$ a proper lsc function such that $K \subseteq {\rm dom}\,h$. Suppose that $\overline{x} \in {\rm argmin}_{K}\,h$, $K$ is star-shaped at $\overline{x}$ and $h$ is strongly star quasiconvex function on $K$ with modulus $\gamma > 0$. Take $\beta^{\prime} > 0$ and a sequence of positive numbers $\{\beta_{k}\}_{k \in \mathbb{N}}$ such that
\begin{equation}\label{param:cond1}
 0 < \beta^{\prime} \leq \beta_{k}, ~ \forall ~ k \in \mathbb{N}_{0}.
\end{equation}
 Then, the sequence $\{x^{k}\}_{k \in \mathbb{N}}$, ge\-ne\-ra\-ted by Algorithm \ref{ppa:sqcx}, is a mi\-ni\-mi\-zing sequence of $h$ on $K$ and converges linearly to the unique solution $\overline{x} = {\rm argmin}_{K}\,h$ with rate of at least
 \begin{equation}\label{linear:rate}
  \frac{1}{1 + \beta^{\prime} \gamma} \in \, ]0, 1[.
 \end{equation}
\end{theorem}

\begin{proof}
 If $x_{k+1}=x_{k}$ for some $k\in \mathbb{N}$, then the algorithm stops and
 $x_{k}\in \mathrm{argmin}_{K}\,h$ by Proposition \ref{prop:fixed}. So we assume
 that $x_{k+1}\neq x_{k}$ for all $k\in \mathbb{N}$.

 Since $x^{k+1} \in {\rm Prox}_{\beta_{k} h} (K, x^{k})$, taking $z=x^{k}$, $x^{*}=x^{k+1}$ and $x=x^{k}$ in \eqref{lin:conv1}, and since we are assuming $x^{k+1} \neq x^{k}$, we obtain
\begin{align}\label{tel:one}
 h(x^{k+1}) < h(x^{k+1}) + \frac{1}{2 \beta_{k}} \lVert x^{k+1} - x^{k} \Vert^{2} \leq h(x^{k}). 
\end{align}
Hence, $\{h(x^{k})\}_{k\in \mathbb{N}_{0}}$ is a decreasing sequence.

Since $x^{k+1} \in \mathrm{Prox}_{\beta_{k} h} (K, x^{k})$, by Proposition \ref{prox:sqcx}, we have
\begin{align}
 \langle x^{k+1} - x^{k}, x^{k+1} - \overline{x} \rangle & \leq -\frac{\gamma \beta_{k}}{2} \lVert x^{k+1} - \overline{x} \rVert^{2}, ~ \forall ~ y \in K. \label{key:01}
\end{align}
 
 On the other hand, 
 \begin{align}
  \lVert x^{k+1} - \overline{x} \rVert^{2} & = \lVert x^{k+1} - x^{k} + x^{k} - \overline{x} \rVert^{2} \notag \\
  & = \lVert x^{k+1} - x^{k} \rVert^{2} + \lVert x^{k} - \overline{x} \rVert^{2} + 2\langle x^{k+1} - x^{k}, x^{k} - \overline{x} \rangle \notag \\
  & = \lVert x^{k} - \overline{x} \rVert^{2} -\lVert x^{k+1} - x^{k} \rVert^{2} + 2 \langle x^{k+1} - x^{k}, x^{k+1} - \overline{x} \rangle. \notag 
 \end{align}
Then, using \eqref{key:01}, 

$(1 + \beta_{k} \gamma) \lVert x^{k+1} - \overline{x} \rVert^{2} \leq \lVert x^{k} - \overline{x} \rVert^{2} - \lVert x^{k+1} - x^{k} \rVert^{2},$

\noindent and so $\lVert x^{k+1} - \overline{x} \rVert^{2} \leq \frac{1}{(1 + \beta^{\prime} \gamma) } \lVert x^{k} - \overline{x} \rVert^{2}$. 

\noindent Hence, $\{x^{k}\}_{k}$ converges linearly to the unique solution $\overline{x} \in {\rm argmin}_{K}\,h$, with convergence rate of at least $\frac{1}{1 + \beta^{\prime} \gamma} \in \, ]0, 1[$, and using \eqref{tel:one} 
we conclude that it is a minimizing sequence, i.e., $h(x^{k}) \downarrow \min_{x \in K}\,h(x)$, which completes the proof.
\end{proof}

As a corollary, we obtain linear convergence for the PPA applied to strongly quasiconvex functions, a results which was not included in \cite{Lara-9}.

\begin{corollary}
 Let $K \subseteq \mathbb{R}^{n}$ be a closed and convex set, $h: \mathbb{R}^{n} \rightarrow \overline{\mathbb{R}}$ a pro\-per, lsc, and strongly quasiconvex function on $K \subseteq {\rm dom}\,h$ with modulus
 $\gamma > 0$. Let $\beta^{\prime} > 0$ and $\{\beta_{k}\}_{k \in \mathbb{N}}$ be a sequence of positive numbers such that \eqref{param:cond1} holds. Then, the sequence $\{x^{k}\}_{k \in \mathbb{N}}$, ge\-ne\-ra\-ted by relation \eqref{step:sqcx}, is a mi\-ni\-mi\-zing sequence of $h$ and converges linearly to the unique solution $\overline{x} = {\rm argmin}_{K}\,h$ with rate of at least
 \begin{equation}
  \frac{1}{1 + \beta^{\prime} \gamma} \in \, ]0, 1[. \notag
 \end{equation}
\end{corollary}

The scope of Theorem \ref{conver:sqcx} is analyzed below.

\begin{remark}\label{rem:compa1}
 \begin{itemize}
  \item[$(a)$] Since $\beta^{'} > 0$ in \eqref{param:cond1}, the linear rate in 
  \eqref{linear:rate} can be as fast as we want, a situation which also happens in the strongly convex case. Furthermore, the linear rate for strongly star quasiconvex functions $\frac{1}{1 + \beta^{\prime} \gamma}$ is exactly as good as in the strongly convex case (see, for instance, \cite[Example 23.40$(ii)$]{BC-2}). 
  
  \item[$(b)$] The function in Example \ref{exam:02} is strongly star quasiconvex being neither weakly convex nor DC nor strongly quasiconvex nor strongly quasar-convex, i.e., the results from \cite{AN,ABS,BB,BLL,CP,IPS,Lara-9,PEN,SSC} and re\-fe\-ren\-ces therein cannot be applied to this function.

  \item[$(c)$] In contrast to strongly convex and strongly quasar-convex functions (see \cite{rock-1976} and \cite{BLL}, respectively), we cannot provide convergence rate for the functions values $\{h(x^{k})\}_{k}$. This limitation is of course expected, since we are including strongly quasiconvex functions, a case in which the analysis of the function values is not possible.
 \end{itemize}
\end{remark}

\begin{remark}
To the best of our knowledge, Theorem \ref{conver:sqcx} provides the first result guaranteeing linear convergence of the sequence generated by thev PPA on closed star-shaped sets, without requiring convexity. It is important to note that this is primarily a theoretical advancement. In practice, the implementation of such an algorithm remains challenging, as closed-form expressions for the proximity operator on star-shaped sets, as well as efficient subroutines for minimization over them are largely unexplored.
\end{remark}

\section{Conclusions and Future Work}\label{sec:06}

We unify the analysis of linear convergence for first-order methods on ge\-ne\-ra\-li\-zed convexity theory by introducing the class of (strongly) star quasiconvex functions. We provide a comprehensive study of this class, presenting several characterizations and useful properties for both nonsmooth and differentiable cases. Our theoretical results have promising implications for nonconvex optimization and open the door to further developments in computational methods on star-shaped sets.

We hope this work provides new avenues for studying the theoretical pro\-per\-ties of functions with star-shaped sublevel sets and their applications in more sophisticated methods, such as splitting-type algorithms (see, for instance, \cite{ABS, BB}). It also lays the groundwork for exploring potential connections between star quasiconvex functions and reference point theory in Prospect Theory (see \cite{KT, TK}).

\section{Declarations}







\subsection{Availability of supporting data}

No data sets were generated during the current study. 

\subsection{Author Contributions}

 All authors contributed equally to the study conception, design and implementation and wrote and corrected the manuscript.

\subsection{Conflicts of Interest}

There are no conflicts of interest or competing interests related to this manuscript.

\subsection{Funding}

This work has been partially su\-ppor\-ted by NAFOSTED---Vietnam through grant 101.01-2025.61 (Khanh), and by
ANID---Chile through Fondecyt Regular 1241040 (Lara).

\subsection{Acknowledgments}

The author wishes to thank both reviewers for their comments and remarks that helped to improve the quality of this paper. Parts of this work were carried out when F. Lara was vi\-si\-ting the Ton Duc Thang University (TDTU) in Ho Chi Minh City, Vietnam, during April 2025, when P.Q. Khanh was visiting the Universidad de Tarapac\'a (UTA), in Arica, Chile, during September 2025. 
The authors wish to thank the above institutions for their hospitality.

\end{document}